\documentclass[12pt]{article}
\usepackage{amsmath,latexsym,amssymb}
\usepackage{enumerate}
\usepackage{amsthm}
\usepackage{epsfig,graphicx}
\usepackage{epic}

\newcommand\NN{\mathbb{N}}
\newcommand\RR{\mathbb{R}}
\newcommand\ZZ{\mathbb{Z}}
\newcommand\A{\mathcal A}

\newcommand\FF{\mathcal{F}}
\newcommand\GG{\mathcal{G}}

\DeclareMathOperator\sP{P}   
\newcommand{\rP}{\mathrm{P}} 

\DeclareMathOperator\dist{dist}

\renewcommand{\mod}{\operatorname{mod}}

\newtheorem{theorem}{Theorem}[section]
\newtheorem{definition}[theorem]{Definition}
\newtheorem{corollary}[theorem]{Corollary}

\newtheorem{lemma}[theorem]{Lemma}
\newtheorem{remark}[theorem]{Remark}
\newtheorem{proposition}[theorem]{Proposition}

\newcommand{\address}{Address: Department of Mathematics, University of North Texas, 1155 Union Circle \#311430, Denton, TX 76203-5017, USA; E-mail: allaart@unt.edu}

\title{On the level sets of the Takagi-van der Waerden functions}
\author{Pieter C. Allaart \footnote{\address}}

\begin{document}

\maketitle

\begin{abstract}

This paper examines the level sets of the continuous but nowhere differentiable functions
\begin{equation*}
f_r(x)=\sum_{n=0}^\infty r^{-n}\phi(r^n x),
\end{equation*}
where $\phi(x)$ is the distance from $x$ to the nearest integer, and $r$ is an integer with $r\geq 2$. It is shown, by using properties of a symmetric correlated random walk, that almost all level sets of $f_r$ are finite (with respect to Lebesgue measure on the range of $f$), but that for an abscissa $x$ chosen at random from $[0,1)$, the level set at level $y=f_r(x)$ is uncountable almost surely. As a result, the occupation measure of $f_r$ is singular.


\end{abstract}

\section{Introduction}

The Takagi-van der Waerden functions are defined by
\begin{equation}
f_r(x):=\sum_{n=0}^\infty r^{-n}\phi(r^n x), \qquad r=2,3,\dots,
\label{eq:Takagi-def}
\end{equation}
where $\phi(x)=\dist(x,\ZZ)$, the distance from $x$ to the nearest integer. The first two are shown in Figure \ref{fig:f2-and-f3}. The case $r=2$ has been studied extensively in the literature; see \cite{AK,Lagarias}. It was introduced in 1903 by Takagi \cite{Takagi}, who proved that it is continuous and nowhere differentiable, and $f_2$ is now generally known as the {\em Takagi function}. Apparently unaware of Takagi's note, Van der Waerden \cite{vdW} in 1930 proved the nondifferentiability of $f_{10}$, using an argument that works for any even $r\geq 4$, but not for odd $r$ or $r=2$. Van der Waerden's paper prompted Hildebrandt \cite{Hildebrandt} and De Rham \cite{deRham} to rediscover the function $f_2$. Billingsley \cite{Billingsley} later gave the simplest proof of nondifferentiability for $r=2$, and it is his argument that is easily extended to all $r$. A short proof is included in Section \ref{sec:nondifferentiability} below.

\begin{figure}
\begin{center}
\epsfig{file=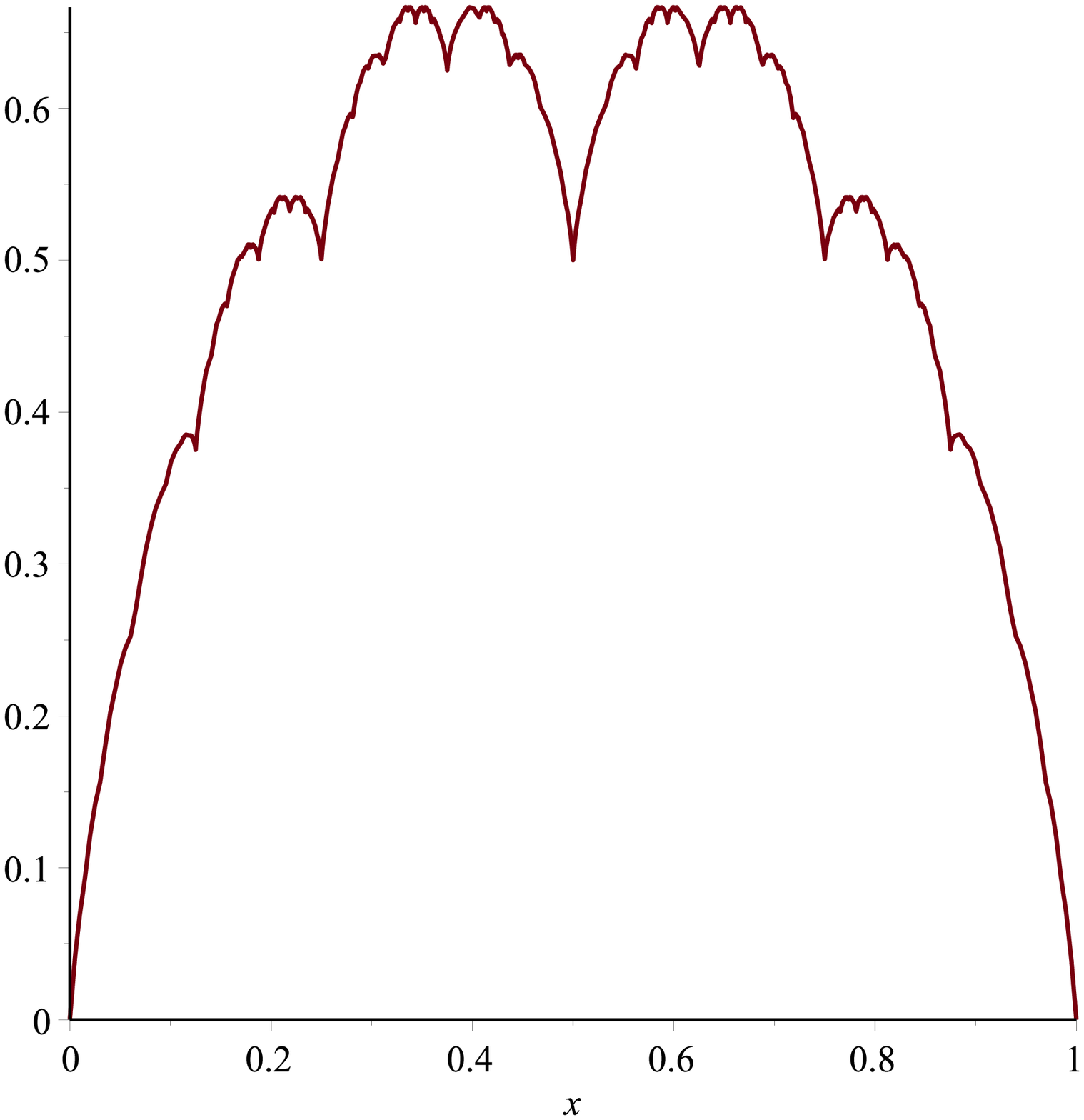, height=.21\textheight, width=.36\textwidth} \qquad
\epsfig{file=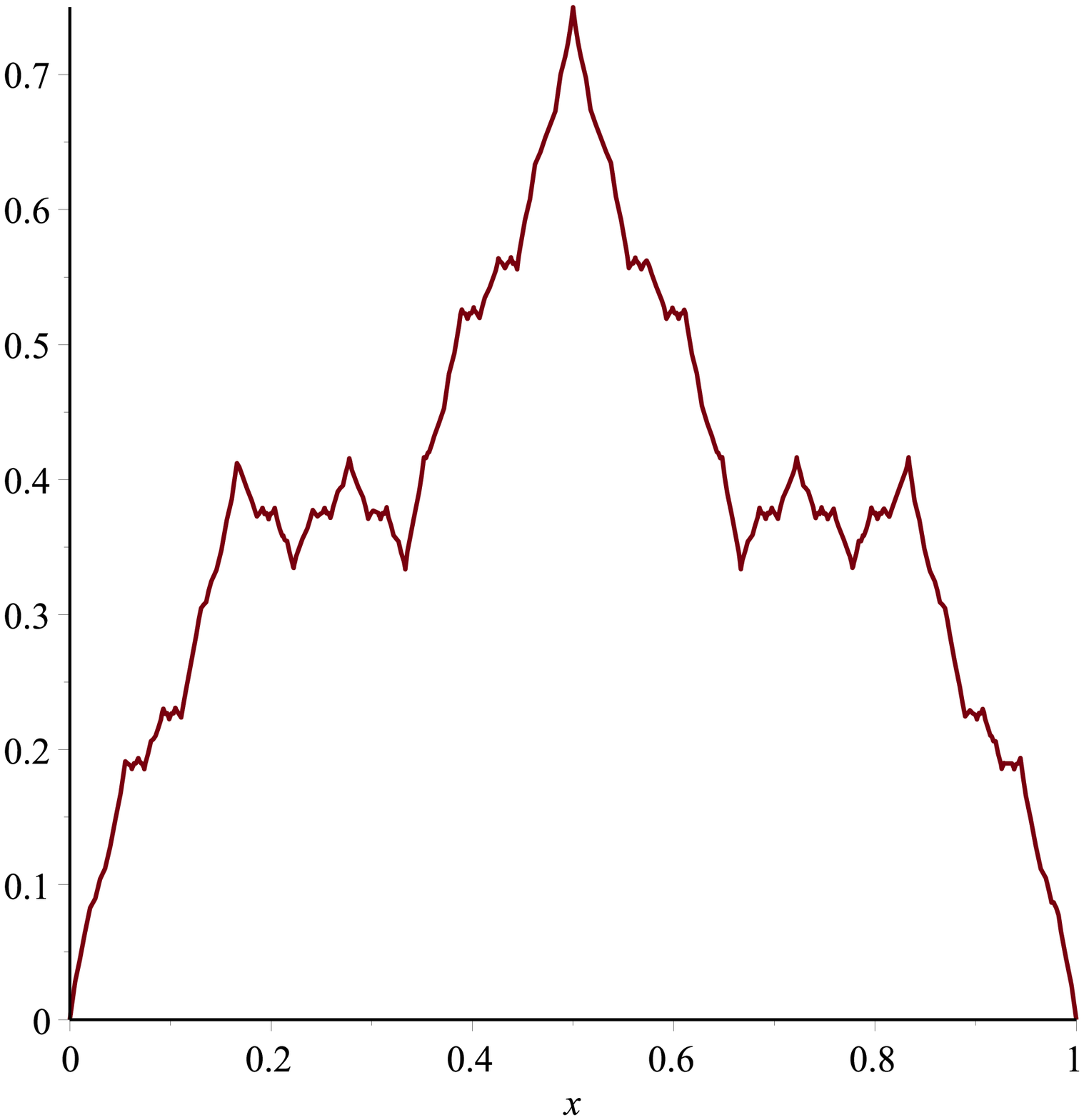, height=.21\textheight, width=.36\textwidth}
\caption{The functions $f_2$ (left) and $f_3$ (right).}
\label{fig:f2-and-f3}
\end{center}
\end{figure}

The main purpose of this note, however, is to prove the following theorem about the level sets of $f_r$, thereby solving Problem 7.6 of \cite{AK}. Let
\begin{equation*}
L_r(y):=\{x\in[0,1): f_r(x)=y\}, \qquad y\in\RR.
\end{equation*}
In what follows, the phrase ``almost every" is always meant in the Lebesgue sense, and $\lambda$ denotes Lebesgue measure on $\RR$.

\begin{theorem} \label{thm:main}
\begin{enumerate}[(i)]
\item For each $r\geq 2$ and for almost every $y$, $L_r(y)$ is finite.
\item For each $r\geq 2$ and for almost every $x\in[0,1)$, $L_r(f_r(x))$ is uncountable.
\end{enumerate}
\end{theorem}

\begin{corollary}
The occupation measure $\mu_{f_r}$ of $f_r$, defined by
\begin{equation*}
\mu_{f_r}(A):=\lambda\{x\in[0,1): f_r(x)\in A\}
\end{equation*}
for Borel sets $A$, is singular with respect to Lebesgue measure $\lambda$.
\end{corollary}

\begin{proof}
Let $A:=\{y\in\RR: |L_r(y)|<\infty\}$. Then $\lambda(\RR\backslash A)=0$ by part (i) of the theorem, whereas $\mu_{f_r}(A)=\lambda\{x\in[0,1): |L_r(f_r(x))|<\infty\}=0$ by part (ii).
\end{proof}

The above results were initially proved by Buczolich \cite{Buczolich} for the case $r=2$ (i.e. the Takagi function); alternative proofs were given recently by the present author \cite{Allaart3} and Lagarias and Maddock \cite{LagMad2}. For even $r\geq 4$, the proof is a fairly straightforward extension of the original arguments, but for odd $r$, the proof involves some additional subtleties; we use results about recurrence of correlated random walks.

It is worth noting that in the Baire category sense, the typical level set of $f_r$ is uncountably infinite. This follows since $f_r$, being nowhere differentiable, is certainly monotone on no interval, and a general theorem of Garg \cite[Theorem 1]{Garg} states that for any continuous function $f$ which is monotone on no interval, the set $\{y: f^{-1}(y)\ \mbox{is a perfect set}\}$ is residual in the range of $f$.

The functions $f_r$ were studied previously by Baba \cite{Baba}, who computed the maxima $M_r:=\max\{f_r(x): 0\leq x\leq 1\}$ and showed that the set $\{x\in[0,1]: f_r(x)=M_r\}$ is a Cantor set of dimension $1/2$ when $r$ is even, but is the singleton $\{1/2\}$ when $r$ is odd. This suggests that the even and odd cases are fundamentally different, an issue that will be encountered again in the proof of Theorem \ref{thm:main}(i) below.

Finally, it was shown by De Amo et al. \cite{ABDF} that the Hausdorff dimension of {\em any} level set of $f_2$ is at most $1/2$. It is straightforward to extend this result to $f_r$ for even $r$, but whether it is true for odd $r$ as well remains unclear, though it is relatively easy to see that $f_r$ always has level sets of dimension {\em at least} $1/2$.

The following notation is used throughout this paper: $\NN$ denotes the set of positive integers, and $\ZZ_+$ the set of nonnegative integers. For a set $E$, $|E|$ denotes the number of elements of $E$. 

\section{Proof of nondifferentiability} \label{sec:nondifferentiability}

For completeness, we first give a short proof of the nondifferentiability of $f_r$.

\begin{theorem} \label{thm:non-diff}
For each $r\geq 2$, $f_r$ has a well-defined finite derivative at no point.
\end{theorem}

\begin{proof}
Following Billingsley \cite{Billingsley}, let $\phi_k(x):=r^{-k}\phi(r^k x)$, $k\in\ZZ_+$, and let $\phi_k^+(x)$ denote the right-hand derivative of $\phi_k$ at $x$. Then $\phi_k^+$ is well defined and $\{-1,1\}$-valued. Fix $x\in[0,1)$, and for each $n\in\NN$, let $u_n=j_n/2r^n$ and $v_n=(j_n+1)/2r^n$, where $j_n\in\ZZ_+$ is chosen so that $u_n\leq x<v_n$. Note that for $k\leq n$, $\phi_k$ is linear on $[u_n,v_n]$, so
\begin{equation*}
\frac{\phi_k(v_n)-\phi_k(u_n)}{v_n-u_n}=\phi_k^+(x), \qquad 0\leq k\leq n.
\end{equation*}
Suppose first that $r$ is even; then $\phi_k(v_n)-\phi_k(u_n)=0$ for $k>n$, and so
\begin{equation*}
m_n:=\frac{f_r(v_n)-f_r(u_n)}{v_n-u_n}=\sum_{k=0}^\infty \frac{\phi_k(v_n)-\phi_k(u_n)}{v_n-u_n}=\sum_{k=0}^{n} \phi_k^+(x).
\end{equation*}
This implies $m_{n+1}-m_n=\pm\,1$, and so $m_n$ cannot have a finite limit. Suppose next that $r$ is odd. Then
for $k\geq n$,
\begin{equation*}
\phi_k(v_n)-\phi_k(u_n)=\frac{(-1)^{j_n}}{2r^k},
\end{equation*}
using the $1$-periodicity of $\phi$. Thus, since $v_n-u_n=1/2r^n$,
\begin{equation*}
m_n:=\frac{f_r(v_n)-f_r(u_n)}{v_n-u_n}
=\sum_{k=0}^{n-1} \phi_k^+(x)+(-1)^{j_n}\sum_{k=n}^\infty \frac{1}{r^{k-n}}
=\sum_{k=0}^{n-1} \phi_k^+(x)+(-1)^{j_n}\frac{r}{r-1}.
\end{equation*}
It is now clear that $m_{n+1}-m_n$ can take only finitely many values, zero not among them, and hence $m_n$ can not have a finite limit as $n\to\infty$. Therefore, $f_r$ does not have a finite derivative at $x$.
\end{proof}

\begin{remark}
{\rm
The above argument is essentially the same as that given in the senior thesis of K. Spurrier \cite{Spurrier}, where a slightly more general result is proved. It is presented here only in a more concise form.
}
\end{remark}

\section{The correlated random walk}

The proof of the main theorem is based on the fact that the slopes of the partial sums of the series \eqref{eq:Takagi-def} follow a correlated random walk when $x$ is chosen at random from $[0,1)$. This is made precise below. For $n\in\ZZ_+$, define the partial sum
\begin{equation*}
f_r^n(x):=\sum_{k=0}^{n-1}r^{-k}\phi(r^k x),
\end{equation*}
and let $s_n(x)$ denote the right-hand derivative of $f_r^n$ at $x$, for $x\in[0,1)$.

\begin{definition} \label{def:CRW}
{\rm
A (symmetric) {\em correlated random walk} with parameter $p$ is a stochastic process $\{S_n\}_{n\in\ZZ_+}$ defined on a probability space $(\Omega,\FF,\rP)$ such that $S_0\equiv 0$ and $S_{n}=S_{n-1}+X_n$ for $n\geq 1$, where $X_1,X_2,\dots$ are $\{-1,1\}$-valued random variables satisfying
$\rP(X_1=1)=\sP(X_1=-1)=1/2$,
and
\begin{equation*}
\rP(X_{n+1}=1|X_n=1)=\sP(X_{n+1}=-1|X_n=-1)=p, \qquad n\in\NN.
\end{equation*}
}
\end{definition}

Observe that when $p=1/2$, the correlated random walk is just a symmetric simple random walk. Correlated random walks have been used to model a wide variety of phenomena, and there is a substantial body of literature about them; a comprehensive list of references can be found in \cite{Chen-Renshaw}.

\begin{lemma} \label{lem:CRW}
Viewed as a stochastic process on the space $(\Omega,\FF,\rP)=([0,1),\mbox{Borels},\lambda)$, the sequence $\{s_n(x)\}_{n\in\ZZ_+}$ is a correlated random walk with parameter
\begin{equation}
p_r=\begin{cases}
1/2, & \mbox{if $r$ is even}\\
\displaystyle\frac{r+1}{2r}, & \mbox{if $r$ is odd}.
\end{cases}
\label{eq:CRW-parameter}
\end{equation}
\end{lemma}

\begin{proof}
With the notation from the proof of Theorem \ref{thm:non-diff}, we have
\begin{equation*}
s_{n+1}(x)=s_n(x)+\phi_n^+(x), \qquad n\in\ZZ_+.
\end{equation*}
Note that $\phi_n^+(x)=1$ if and only if $0\leq r^n x\mod 1<1/2$, which is the case if and only if
\begin{equation*}
2x\in\bigcup_{k\in\ZZ} \left[\frac{2k}{r^n},\frac{2k+1}{r^n}\right).
\end{equation*}
For fixed $k$, the interval $[2k/r^n,(2k+1)/r^n)$ has length $r^{-n}$ and partitions into the $r$ subintervals
\begin{equation*}
\left[\frac{2kr}{r^{n+1}},\frac{2kr+1}{r^{n+1}}\right), \left[\frac{2kr+1}{r^{n+1}},\frac{2kr+2}{r^{n+1}}\right), \dots, \left[\frac{(2k+1)r-1}{r^{n+1}},\frac{(2k+1)r}{r^{n+1}}\right).
\end{equation*}
If $r$ is even, half of these are of the form $[2j/r^{n+1},(2j+1)/r^{n+1})$ with $j\in\ZZ$; if $r$ is odd, $(r+1)/2$ of them are of this form. Thus, if $x$ is chosen at random from $[0,1)$, we have
\begin{equation}
\rP(\phi_{n+1}^+(x)=1|\phi_n^+(x)=1)=p_r.
\label{eq:plus-plus}
\end{equation}
The equality
\begin{equation}
\rP(\phi_{n+1}^+(x)=-1|\phi_n^+(x)=-1)=p_r
\label{eq:minus-minus}
\end{equation}
can be established by a very similar argument. But it also follows from \eqref{eq:plus-plus} and symmetry. More precisely, since $\phi_n$ is symmetric about $x=1/2$, we have that 
\begin{equation}
x\not\in \left\{\frac{j}{2r^n}: j\in\ZZ\right\} \quad \Rightarrow \quad \phi_n^+(x)=-\phi_n^+(1-x).
\label{eq:skew-symmetry}
\end{equation}
Since the mapping $x\mapsto 1-x$ preserves Lebesgue measure and the set $\{j/2r^n: j\in\ZZ\}$ is Lebesgue null for each $n$, \eqref{eq:plus-plus} and \eqref{eq:skew-symmetry} together imply \eqref{eq:minus-minus}.
\end{proof}

\section{Level sets: Proof of the main theorem}

For $n\in\NN$ and $j=0,1,\dots,2r^{n-1}-1$, let $I_{n,j}$ denote the interval
\begin{equation*}
I_{n,j}:=\left[\frac{j}{2r^{n-1}},\frac{j+1}{2r^{n-1}}\right),
\end{equation*}
and let $s_{n,j}$ denote the slope of $f_r^n$ on $I_{n,j}$. In other words, $s_{n,j}=s_n(x)$ if $x\in I_{n,j}$. Note that $s_{n,j}$ is well defined since $f_r^n$ is linear on $I_{n,j}$. For $x\in[0,1)$ and $n\in\NN$, let $I_n(x)$ denote that interval $I_{n,j}$ which contains $x$.

A fundamental difference between the case of even $r$ and the case of odd $r$ is the following: $f_r(j/2r^{n-1})=f_r^n(j/2r^{n-1})$ for all $j\in\ZZ$ when $r$ is even, but this last equality holds only for even $j$ when $r$ is odd. Hence, $f_r$ corresponds with $f_r^n$ at both endpoints of $I_{n,j}$ in the case of even $r$, but at only one of the endpoints in the case of odd $r$. In fact, when $r$ and $j$ are both odd, 
$f_r(j/2r^{n-1})>f_r^k(j/2r^{n-1})$ for every $k\geq 0$, so the value of $f_r$ at $j/2r^{n-1}$ does not get ``fixed" after a finite number of steps in the construction of $f_r$. This is the main reason why understanding the level sets of $f_r$ is harder when $r$ is odd.

Let 
\begin{equation*}
\GG_l:=\{(x,f_r(x)): 0\leq x\leq 1/2\},
\end{equation*}
so $\GG_l$ is the ``left half" of the graph of $f_r$ over $[0,1]$. Our first goal is to prove a kind of self-similarity result, namely that the graph of $f_r$ above any interval $I_{n,j}$ with $s_{n,j}=0$ consists of $r$ nonoverlapping similar copies of $\GG_l$, scaled by $r^{-n}$ and with alternating orientations (see Corollary \ref{cor:self-similar-limit} below). This will follow from an analogous statement about the partial sums $f_r^n$ in Lemma \ref{lem:self-similar-partial} below. Observe that an $n$th level interval $I_{n,j}$ decomposes into $(n+1)$th level intervals as
\begin{equation*}
I_{n,j}=\bigcup_{i=0}^{r-1} I_{n+1,rj+i}.
\end{equation*}

\begin{lemma} \label{lem:self-similar-partial}
If $s_{n,j}=0$, then for each $m>n$, for $i=0,1,\dots,r-1$ and for $x\in \bar{I}_{n+1,rj+i}$ (where $\bar{I}$ denotes the closure of $I$),
\begin{equation}
f_r^m(x)-f_r^n\left(\frac{j}{2r^{n-1}}\right)=\begin{cases}
r^{-n}f_r^{m-n}(x'), & \mbox{if $rj+i$ is even}\\
r^{-n}f_r^{m-n}\left(\frac12-x'\right), & \mbox{if $rj+i$ is odd},
\end{cases}
\label{eq:self-similar-finite}
\end{equation}
where $x'$ is the point in $[0,1/2]$ such that
\begin{equation}
x=\frac{rj+i}{2r^n}+\frac{x'}{r^n}.
\label{eq:x-affine-map}
\end{equation}
\end{lemma}

\begin{proof}
If $s_{n,j}=0$, then $f_r^n$ is constant on $I_{n,j}$, and so
\begin{equation*}
f_r^m(x)-f_r^n\left(\frac{j}{2r^{n-1}}\right)=f_r^m(x)-f_r^n(x)=\sum_{k=n}^{m-1} r^{-k}\phi(r^k x).
\end{equation*}
If $rj+i$ is even, then $k\geq n$ implies, by the $1$-periodicity of $\phi$,
\begin{equation*}
\phi(r^k x)=\phi\left(r^{k-n}\left(\frac{rj+i}{2}+x'\right)\right)=\phi(r^{k-n}x'),
\end{equation*}
so that
\begin{equation*}
\sum_{k=n}^{m-1} r^{-k}\phi(r^k x)=\sum_{k=n}^{m-1} r^{-k}\phi(r^{k-n}x')=r^{-n}f_r^{m-n}(x').
\end{equation*}
Assume then that $rj+i$ is odd. If $r$ is odd as well, then the symmetry of $\phi$ gives, for all $k\geq n$,
\begin{equation*}
\phi\left(r^{k-n}\left(\frac{rj+i}{2}+x'\right)\right) = \phi\left(\frac12+r^{k-n}x'\right)
=\phi\left(\frac12-r^{k-n}x'\right)=\phi\left(r^{k-n}\left(\frac12-x'\right)\right),
\end{equation*}
and so
\begin{equation*}
\sum_{k=n}^{m-1} r^{-k}\phi(r^k x)=r^{-n}f_r^{m-n}\left(\frac12-x'\right).
\end{equation*}
On the other hand, if $r$ is even, then $r^{k-n}/2\in\ZZ$ for all $k>n$, and we get (with the convention that the empty sum takes the value zero):
\begin{align*}
\sum_{k=n}^{m-1} r^{-k}\phi(r^k x) &= r^{-n}\phi\left(\frac12+x'\right)+\sum_{k=n+1}^{m-1} r^{-k}\phi(r^{k-n}x')\\
&=r^{-n}\phi\left(\frac12-x'\right)+\sum_{k=n+1}^{m-1} r^{-k}\phi\left(r^{k-n}\left(\frac12-x'\right)\right)\\
&=r^{-n}f_r^{m-n}\left(\frac12-x'\right).
\end{align*}
This completes the proof.
\end{proof}

Taking limits as $m\to\infty$ in \eqref{eq:self-similar-finite}, we immediately obtain:

\begin{corollary} \label{cor:self-similar-limit}
If $s_{n,j}=0$, then for $i=0,1,\dots,r-1$ and $x\in \bar{I}_{n+1,rj+i}$,
\begin{equation*}
f_r(x)-f_r^n\left(\frac{j}{2r^{n-1}}\right)=\begin{cases}
r^{-n}f_r(x'), & \mbox{if $rj+i$ is even}\\
r^{-n}f_r\left(\frac12-x'\right), & \mbox{if $rj+i$ is odd},
\end{cases}
\end{equation*}
where $x'$ is defined by \eqref{eq:x-affine-map}.
\end{corollary}

\begin{proof}[Proof of Theorem \ref{thm:main}(ii)]
Let
\begin{equation*}
\Omega_0:=\{x\in[0,1): s_n(x)=0\ \mbox{for infinitely many $n$}\}.
\end{equation*}
We claim that for each $x\in\Omega_0$, the level set $L_r(f_r(x))$ is uncountably infinite. Since the symmetric correlated random walk is recurrent (see, for instance, Example 1 on p.~879 of \cite{Chen-Renshaw}), $\Omega_0$ is of full measure in $[0,1)$ and so part (ii) of the theorem will follow from the claim and Lemma \ref{lem:CRW}. 

To prove the claim, the following notation is helpful. Let $\Sigma:=\{1,2,\dots,r\}$, let $\Sigma^*:=\bigcup_{k=0}^\infty \Sigma^k$ where $\Sigma^0:=\{\emptyset\}$, and denote by $\Sigma^\NN$ the set of all infinite sequences $(i_1,i_2,\dots)$ such that $i_k\in\Sigma$ for all $k\in\NN$. Let $x\in\Omega_0$, and let $\{n_k: k=0,1,\dots\}$ be a strictly increasing sequence of positive integers such that $s_{n_k}(x)=0$ for all $k$. Repeated application of Lemma \ref{lem:self-similar-partial} shows that there is an $r$-ary tree of intervals $I_{\bf i}$, ${\bf i}\in \Sigma^*$ with the following properties:
\begin{enumerate}
\item $I_\emptyset=I_{n_0}(x)$;
\item if ${\bf i},{\bf j}\in \Sigma^k$ with ${\bf i}\neq {\bf j}$, then $I_{\bf i}\cap I_{\bf j}=\emptyset$;
\item for each ${\bf i}\in\Sigma^k$, $f_r^{n_k}$ is constant on $I_{\bf i}$ with value $y_k:=f_r^{n_k}(x)$;
\item for each ${\bf i}\in\Sigma^k$, there is $j\in\ZZ$ such that $I_{\bf i}=I_{n_k,j}$;
\item if ${\bf i}=(i_1,\dots,i_k)\in\Sigma^k$ and ${\bf j}=(i_1,\dots,i_k,j)$ for any $j\in\Sigma$, then $I_{\bf j}\subset I_{\bf i}$.
\end{enumerate}
For any sequence ${\bf i}=(i_1,i_2,\dots)\in\Sigma^\NN$, the intersection $\bigcap_{k=1}^\infty I_{i_1,\dots,i_k}$ is a single point $x_{\bf i}$, and if ${\bf i}$ and ${\bf j}$ are different members of $\Sigma^\NN$, then $x_{\bf i}\neq x_{\bf j}$. Moreover, for each ${\bf i}\in\Sigma^\NN$, $f_r(x_{\bf i})=\lim_{k\to\infty}y_k=f_r(x)$. Thus, $L_r(f_r(x))$ is uncountable, and the claim is established.
\end{proof}

For the Takagi function (i.e. the case $r=2$), the sequence $\{s_n(x)\}_n$ uniquely determines $x$, and the sequence $\{|s_n(x)|\}_n$ determines $f_2(x)$ (see, for instance, \cite[Lemma 2.1]{Allaart3}). However, neither statement holds when $r\geq 3$. Instead of the simple Lemma 2.1 of \cite{Allaart3}, we need the carefully constructed mapping $\rho$ in Lemma \ref{lem:finite-to-one-map} below. This requires some additional notation. Let
\begin{equation*}
C:=\left\{\frac{j}{2r^n}: n\in\ZZ_+, j=0,1,\dots,2r^n-1\right\},
\end{equation*}
and
\begin{equation*}
D:=[0,1)\backslash C,
\end{equation*}
so $C$ is the set of ``corner" points of the partial sums $f_r^n$, and $D$, its complement, is the set of points at which each partial sum $f_r^n$ has a well-defined two-sided derivative. For $x\in[0,1)$, define
\begin{equation*}
n_+(x):=\inf\{n:s_n(x)<0\}-1=\sup\{n:s_1(x)\geq 0,\dots,s_n(x)\geq 0\}.
\end{equation*}
Let
\begin{align}
D_{fin}:&=\{x\in D: s_n(x)=0\ \mbox{for only finitely many $n$}\}, \label{eq:D-fin}\\
D_{fin}^+:&=\{x\in D_{fin}: s_n(x)\geq 0\ \mbox{for every $n$}\}. \label{eq:D-fin-plus}
\end{align}

\begin{lemma} \label{lem:finite-to-one-map}
There is a mapping $\rho:D\to D$ such that:
\begin{enumerate}[(i)]
\item $f_r(\rho(x))=f_r(x)$ for all $x$;
\item $|s_n(\rho(x))|=|s_n(x)|$ for all $n$ and all $x$;
\item $\rho(x)=x$ if and only if $s_n(x)\geq 0$ for all $n$;
\item if $s_n(x)<0$ for some $n$, then $n_+(\rho(x))>n_+(x)$;
\item $\rho(D_{fin})\subset D_{fin}$;
\item the restriction of $\rho$ to $D_{fin}$ is finite-to-one;
\item the limit $\rho^\infty(x):=\lim_{n\to\infty}\rho^n(x)$ exists for every $x$, and $s_n(\rho^\infty(x))\geq 0$ for all $n$ and all $x$; here $\rho^n$ denotes $n$-fold iteration of $\rho$.
\end{enumerate}
\end{lemma}

\begin{proof}
We first construct $\rho$. If $s_n(x)\geq 0$ for every $n$, set $\rho(x):=x$. Properties (i)-(iv) are trivially satisfied for such a point $x$. Assume now that $s_n(x)<0$ for some $n$; then $n_0:=n_+(x)<\infty$, and $s_{n_0}(x)=0$. If $n_0=0$, then $s_1(x)<0$ and so $x\in(1/2,1)$, since $x\not\in C$; in this case we put $\rho(x):=1-x$. Properties (i)-(iv) are easily checked in this case as well. 

Suppose then that $n_0\geq 1$. Let $j_0$ be the integer such that $I_{n_0}(x)=I_{n_0,j_0}$. 
If $x\in I_{n_0+1,rj_0}$, put
\begin{equation*}
\rho(x):=\frac{rj_0+1}{r^{n_0}}-x,
\end{equation*}
the reflection of $x$ across the right endpoint of $I_{n_0+1}(x)$. Otherwise, $x\in I_{n_0+1,rj_0+l}$ for some $l\in\{1,\dots,r-1\}$, and we put
\begin{equation*}
\rho(x):=\frac{rj_0+l}{r^{n_0}}-x,
\end{equation*}
the reflection of $x$ across the {\em left} endpoint of $I_{n_0+1}(x)$. 

The construction of $\rho(x)$ is illustrated in Figure \ref{fig:proof-illustration}, in which $x\in I_{n_0+1,rj_0}$, whereas $x'\in I_{n_0+1,rj_0+2}$ and $x''\in I_{n_0+1,rj_0+4}$. Note that in all cases when $n_0<\infty$, we have that $\rho(x)\in I_{n_0}(x)$ and $s_{n_0+1}(\rho(x))>0$, which gives property (iv). Since $x\not\in C$, it is clear that $\rho(x)\neq x$, so (iii) holds. Property (i) follows by Corollary \ref{cor:self-similar-limit}, and property (ii) follows by Lemma \ref{lem:self-similar-partial}. 

The construction of $\rho$ is complete, and it remains to verify properties (v), (vi) and (vii).
Note that (v) follows immediately from (ii). To see (vi), let $z\in D_{fin}$, and let $n(z):=\max\{n\geq 0: s_n(z)=0\}$. Suppose $x\in D_{fin}$ with $\rho(x)=z$, and assume $x\neq z$. Then $n_+(x)\leq n(z)$, so there are only finitely many possible values for the number $n_0$ in the construction of $\rho(x)$ above. If $n_0=0$, there is only one $x$ with $n_+(x)=n_0$ such that $\rho(x)=z$, since $\rho(x)=1-x$ in this case. If $n_0\geq 1$, then the set $\{x: n_+(x)=n_0\ \mbox{and}\ \rho(x)=z\}$ can consist of two points (one reflected from the left of $z$, the other reflected from the right; see Figure \ref{fig:proof-illustration}), but no more. Thus, $\rho$ is finite-to-one on $D_{fin}$.

Property (vii) is obvious if there exists $k$ such that $s_n(\rho^k(x))\geq 0$ for all $n$. If there is no such $k$, then $n_+(\rho^k(x))$ is strictly increasing in $k$ by (iv), and since $\rho(x)\in I_{n_+(x)}(x)$ for every $x$, the intervals
\begin{equation*}
I_{n_+(\rho^k(x))}(\rho^k(x)), \qquad k\in\NN
\end{equation*}
are nested and shrink to a single point. This point must be the limit of $\rho^k(x)$.
\end{proof}

\begin{figure}
\begin{center}
\begin{picture}(360,175)(0,0)
\put(30,60){\line(1,0){300}}
\put(28,66){\makebox(0,0)[tl]{$[$}}
\put(328,66){\makebox(0,0)[tl]{$)$}}
\put(340,65){\makebox(0,0)[tl]{$I_{n_0,j_0}$}}
\multiput(30,100)(5,0){60}{\line(1,0){3}}
\dashline[+20]{3}(30,100)(0,70)
\dashline[+20]{3}(330,100)(360,130)
\put(0,110){\makebox(0,0)[tl]{$f_r^{n_0}$}}
\put(90,100){\line(-1,1){60}}
\put(90,100){\line(1,1){60}}
\put(210,100){\line(-1,1){60}}
\put(210,100){\line(1,1){60}}
\put(330,100){\line(-1,1){60}}
\put(0,170){\makebox(0,0)[tl]{$f_r^{n_0+1}$}}
\put(54,135){\makebox(0,0)[tl]{$\bullet$}}
\put(120,135){\makebox(0,0)[tl]{$\bullet$}}
\put(174,135){\makebox(0,0)[tl]{$\bullet$}}
\put(230,125){\makebox(0,0)[tl]{$\bullet$}}
\put(304,125){\makebox(0,0)[tl]{$\bullet$}}
\put(65,132){\vector(1,0){50}}
\put(88,145){\makebox(0,0)[tl]{$\rho$}}
\put(170,132){\vector(-1,0){40}}
\put(148,127){\makebox(0,0)[tl]{$\rho$}}
\put(300,123){\vector(-1,0){60}}
\put(268,118){\makebox(0,0)[tl]{$\rho$}}
\put(57,57){\line(0,1){6}}
\put(123,57){\line(0,1){6}}
\put(177,57){\line(0,1){6}}
\put(233,57){\line(0,1){6}}
\put(307,57){\line(0,1){6}}
\put(53,48){\makebox(0,0)[tl]{$x$}}
\put(95,52){\makebox(0,0)[tl]{$\rho(x)=\rho(x')$}}
\put(173,52){\makebox(0,0)[tl]{$x'$}}
\put(222,52){\makebox(0,0)[tl]{$\rho(x'')$}}
\put(301,52){\makebox(0,0)[tl]{$x''$}}
\put(30,25){\line(1,0){60}}
\put(150,25){\line(1,0){60}}
\put(270,25){\line(1,0){60}}
\put(28,31){\makebox(0,0)[tl]{$[$}}
\put(88,31){\makebox(0,0)[tl]{$)$}}
\put(148,31){\makebox(0,0)[tl]{$[$}}
\put(208,31){\makebox(0,0)[tl]{$)$}}
\put(268,31){\makebox(0,0)[tl]{$[$}}
\put(328,31){\makebox(0,0)[tl]{$)$}}
\put(40,20){\makebox(0,0)[tl]{$I_{n_0+1}(x)$}}
\put(157,20){\makebox(0,0)[tl]{$I_{n_0+1}(x')$}}
\put(276,20){\makebox(0,0)[tl]{$I_{n_0+1}(x'')$}}
\end{picture}
\caption{An illustration of the proof of Lemma \ref{lem:finite-to-one-map}, for the case $r=5$ and $j_0$ odd. The graph of $f_r^{n_0}$ is drawn in dashed lines, the graph of $f_r^{n_0+1}$ in solid lines. Only points in the leftmost subinterval of $I_{n_0,j_0}$ move to the right under the mapping $\rho$.}
\label{fig:proof-illustration}
\end{center}
\end{figure}
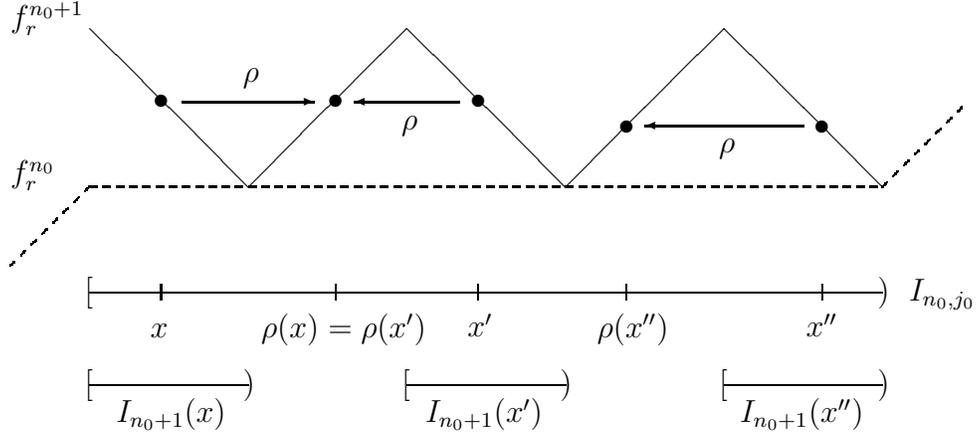

\begin{corollary} \label{cor:finite-to-one-map}
There is a finite-to-one mapping $\pi: D_{fin}\to D_{fin}^+$ such that $f_r(\pi(x))=f_r(x)$ for every $x\in D_{fin}$.
\end{corollary}

\begin{proof}
Let $\rho$ be as in Lemma \ref{lem:finite-to-one-map}. For $x\in D_{fin}$, there is by properties (ii)-(v) an integer $k$ such that $\rho^{k+1}(x)=\rho^k(x)$, and we put $\pi(x):=\rho^k(x)$. By property (i), $f_r(\pi(x))=f_r(x)$. Now fix $z\in D_{fin}^+$, and let $m(z):=\#\{n\geq 0:s_n(z)=0\}$. If $\pi(x)=z$, then $z=\rho^k(x)$ for some $k$, and by property (iv) this $k$ can be taken to be no greater than $m(z)$. Since $\rho$ is finite-to-one on $D_{fin}$, so is $\rho^k$ for each $k$. Thus, there are at most finitely points $x$ such that $\pi(x)=z$, so $\pi$ is finite-to-one.
\end{proof}

\begin{corollary} \label{cor:positive-slopes}
For each $x\in D$, there is a point $x'\in D$ such that $s_n(x')=|s_n(x)|$ for every $n$, and $f_r(x')=f_r(x)$.
\end{corollary}

\begin{proof}
Take $x':=\rho^\infty(x)$. The result follows from properties (i), (ii) and (vii) of Lemma \ref{lem:finite-to-one-map} and the continuity of $f_r$.
\end{proof}

The following lemma underlines the additional subtleties that need to be addressed in case $r$ is odd. Let
\begin{equation}
D^*:=\{x\in D: s_n(x)>0\ \mbox{for every $n\geq 1$}\}.
\label{eq:D-star}
\end{equation}
For two points $x$ and $x'$ with $x<x'$, say an interval $I=[a,b)$ {\em separates} $x$ and $x'$ if $x<a<b<x'$.

\begin{lemma} \label{lem:flat-in-between}
Let $x,x'\in D^*$ with $x<x'$, and suppose that $f_r(x)=f_r(x')$. Then there exists a pair of integers $(n,j)$ such that (i) $I_{n,j}$ separates $x$ and $x'$; (ii) $s_{n,j}=0$; and (iii) $f_r(x)\in f_r(I_{n,j})$.
\end{lemma}

\begin{proof}
There clearly exists a smallest $n$ for which there exists $j$ such that $I_{n,j}$ separates $x$ and $x'$. Fix this $n$ and such a $j$, and let $s:=s_{n,j}$. Observe that $[x,x']\subset I_{n-1}(x)\cup I_{n-1}(x')$, as otherwise there would exist $i$ such that $I_{n-1,i}$ separates $x$ and $x'$. Since $x$ and $x'$ lie in $D^*$, the slopes of $f_r^{n-1}$ are strictly positive on $I_{n-1}(x)\cup I_{n-1}(x')$, and hence, on $[x,x']$. As a result, $f_r^n$ is nondecreasing throughout $[x,x']$, and in particular, $s\geq 0$. We have
\begin{align}
\begin{split}
f_r(x)&=f_r^n(x)+\sum_{k=n}^\infty r^{-k}\phi(r^k x)=f_r^n(x)+r^{-n}f_r(r^n x)\\
&\leq f_r^n\left(\frac{j}{2r^{n-1}}\right)+r^{-n}\max_z f_r(z)
<f_r^n\left(\frac{j}{2r^{n-1}}\right)+r^{-n}.
\end{split}
\label{eq:inequality-chain}
\end{align}
Now if $s\geq 1$, then
\begin{equation*}
f_r^n\left(\frac{j+1}{2r^{n-1}}\right)-f_r^n\left(\frac{j}{2r^{n-1}}\right)=\frac{s}{2r^{n-1}}\geq \frac{1}{2r^{n-1}}\geq r^{-n},
\end{equation*}
so that
\begin{equation*}
f_r(x)<f_r^n\left(\frac{j+1}{2r^{n-1}}\right)\leq f_r^n(x')\leq f_r(x').
\end{equation*}
This contradicts the hypothesis of the lemma, and hence, $s=0$. 

It remains to show that $f_r(x)\in f_r(I_{n,j})$. First, since $s=0$, we have
\begin{equation*}
\max_{u\in I_{n,j}} f_r(u)=f_r^n\left(\frac{j}{2r^{n-1}}\right)+r^{-n}\max_z f_r(z)\geq f_r(x),
\end{equation*}
as in \eqref{eq:inequality-chain}. On the other hand, 
\begin{equation*}
f_r(x)=f_r(x')\geq f_r^n(x')\geq f_r^n\left(\frac{j+1}{2r^{n-1}}\right)=\min_{u\in I_{n,j}} f_r(u),
\end{equation*}
since the value of $f_r$ at one or both endpoints of $I_{n,j}$ must equal the (constant) value of $f_r^n$ on $I_{n,j}$, in view of Corollary \ref{cor:self-similar-limit}.
Thus, $f_r(x)\in f_r(I_{n,j})$, and the proof is complete.
\end{proof}

\begin{remark} \label{rem:only-when-odd}
{\rm
The hypothesis of Lemma \ref{lem:flat-in-between} is never satisfied when $r$ is even; since this fact is not needed to prove Theorem \ref{thm:main}, its (not too difficult) proof is omitted here. When $r$ is odd, however, there may exist points $x$ and $x'$ satisfying the hypothesis of the lemma. For instance, when $r=3$, $x=17/108$ and $x'=37/108$, a straightforward calculation yields that $(s_n(x))_{n\geq 1}=(s_n(x'))_{n\geq 1}=(1,2,3,4,3,4,\dots)$, and $f_3(x)=f_3(x')=3/8$.
}
\end{remark}

\begin{remark}
{\rm
In connection with the previous remark, it is worth noting that $f_r(x)$ is rational when $x$ is rational. This follows since if $x$ is rational, it has a base-$r$ expansion which is eventually periodic, and so $r^n x \mod 1$ is eventually periodic. By the $1$-periodicity of $\phi$, this implies that $\phi(r^n x)$ is eventually periodic, and it is easy to deduce from this that $f_r(x)$ is rational.
}
\end{remark}

Before stating the next important preliminary result, we define the discrete sets
\begin{align*}
\A&:=\{(n,j): n\in\NN,\ 0\leq j<2r^{n-1},\ \mbox{and}\ s_{n,j}=0\},\\
\A^+&:=\{(n,j)\in \A:\ \mbox{the slope of $f_r^k$ is nonnegative on $I_{n,j}$ for $k=1,\dots,n$}\},
\end{align*}
and finally,
\begin{equation*}
\A^+(y):= \{(n,j)\in \A^+: y\in f_r(I_{n,j})\}, \qquad y\in\RR.
\end{equation*}
To visualize the set $\A^+(y)$, it may help to recall from Corollary \ref{cor:self-similar-limit} that if $s_{n,j}=0$, then the graph of $f_r$ above $I_{n,j}$ consists of $r$ small-scale similar copies of the half graph $\GG_l$, sitting side by side at the same height with alternating orientations. The set $\A^+(y)$ consists of those pairs $(n,j)$ in $\A$ for which the horizontal line at level $y$ intersects these $r$ similar copies of $\GG_l$, and, furthermore, the partial sum functions $f_r^1,\dots,f_r^n$ all have nonnegative slope throughout $I_{n,j}$.

\begin{proposition} \label{prop:key-idea}
Let $y\in\RR\backslash f_r(C)$. If $|\A^+(y)|<\infty$, then $|L_r(y)|<\infty$.
\end{proposition}

\begin{proof}
Recall the definitions of $D_{fin}$ and $D_{fin}^+$ from \eqref{eq:D-fin} and \eqref{eq:D-fin-plus}.
Let
\begin{equation*}
D^+:=\{x\in D: s_n(x)\geq 0\ \mbox{for every $n$}\}. 
\end{equation*}
Let $y$ be as given, and suppose $|\A^+(y)|<\infty$. Then $L_r(y)\subset D$, and 
\begin{equation}
L_r(y)\cap D^+\subset D_{fin}. 
\label{eq:inclusion}
\end{equation}
To see \eqref{eq:inclusion}, let $x\in L_r(y)\cap D^+$. Then $f_r(x)=y$ and $s_n(x)\geq 0$ for all $n$. Because $|\A^+(y)|<\infty$, there are only finitely many pairs $(n,j)$ in $\A^+$ such that $x\in I_{n,j}$. Since $s_n(x)\geq 0$ for all $n$, this last statement remains true if we replace $\A^+$ with $\A$. Thus, there are only finitely many $n$ such that $s_n(x)=0$, which means $x\in D_{fin}$.

For arbitrary $x\in L_r(y)$, there is by Corollary \ref{cor:positive-slopes} a point $x'\in L_r(y)\cap D^+$ such that $|s_n(x)|=s_n(x')$ for each $n$. But then $x'\in D_{fin}$ by \eqref{eq:inclusion}, and so $x\in D_{fin}$. Therefore, $L_r(y)\subset D_{fin}$. By Corollary \ref{cor:finite-to-one-map}, it is now enough to show that $|L_r(y)\cap D_{fin}^+|<\infty$. To this end, define an equivalence relation $\sim$ on $D$ as follows. Let
\begin{equation*}
n(x):=\sup\{n\geq 0: s_n(x)=0\}, 
\end{equation*}
and say $x\sim x'$ if either $x=x'$, or all of the following hold:
\begin{enumerate}[(i)]
\item $n(x)=n(x')=:n<\infty$;
\item $I_n(x)=I_n(x')$; and
\item $x-x'=mr^{-n}$ for some $m\in\ZZ$.
\end{enumerate}
Note that, since $I_n(x)$ has length $1/(2r^{n-1})$, condition (ii) forces the number $m$ in (iii) to satisfy $|m|<r/2$, and therefore, $\sim$ has finite equivalence classes.

Let $E$ be the set of those points $x$ in $L_r(y)\cap D_{fin}^+$ which are the leftmost member of their equivalence class; that is,
\begin{equation*}
x\in E\quad \Longleftrightarrow\quad x\in L_r(y)\cap D_{fin}^+\ \mbox{and $x\leq x'$ for each $x'$ with $x\sim x'$}.
\end{equation*}
Since $\sim$ has finite equivalence classes, it suffices to show that $E$ is finite.

\bigskip
\noindent {\bf Claim:} For any two points $x,x'\in E$ with $x<x'$, there is a pair $(n,j)\in \A^+(y)$ such that $I_{n,j}\cap [x,x']\neq\emptyset$, and $I_{n,j}$ contains at most one of $x$ and $x'$.

\bigskip
Assuming the Claim for now, we can finish the proof of the Proposition as follows. Suppose, by way of contradiction, that $E$ is infinite. Then we can find a strictly monotone sequence $\{x_\nu\}$ in $E$; say $\{x_\nu\}$ is increasing. (The argument when $\{x_\nu\}$ is decreasing is entirely similar.) The Claim implies the existence of pairs $\{(n_\nu,j_\nu): \nu\in\NN\}$ in $\A^+(y)$ such that for each $\nu$, $I_{n_\nu,j_\nu}\cap [x_\nu,x_{\nu+1}]\neq\emptyset$ and $I_{n_\nu,j_\nu}$ contains at most one of $x_\nu$ and $x_{\nu+1}$. Consider an arbitrary interval $I_{n,j}$, and suppose $I_{n,j}$ intersects $k$ of the intervals $[x_\nu,x_{\nu+1}]$. Then $I_{n,j}$ must fully contain at least $k-2$ of these, so by the hypothesis about $I_{n_\nu,j_\nu}$, $I_{n,j}$ can occur at most twice in the sequence $\{I_{n_\nu,j_\nu}\}$. Hence, this sequence contains infinitely many distinct intervals. But this implies $|\A^+(y)|=\infty$, contradicting the hypothesis of the proposition. Therefore, $|E|<\infty$.

\bigskip
We now turn to the proof of the Claim. Let $x,x'\in E$ with $x<x'$, and consider three cases:

\bigskip
{\em Case 1:} $s_n(x)>0$ and $s_n(x')>0$ for all $n\geq 1$. By Lemma \ref{lem:flat-in-between}, there is a pair $(n,j)\in \A^+$ such that $I_{n,j}$ separates $x$ and $x'$, and $y\in f_r(I_{n,j})$. Thus, $(n,j)\in \A^+(y)$.

\bigskip
{\em Case 2:} There is a pair $(n,j)\in \A^+$ such that $x\in I_{n,j}$ but $x'\not\in I_{n,j}$, or vice versa. Since $y=f_r(x)=f_r(x')$, we have in either case that $y\in f_r(I_{n,j})$, so $(n,j)\in \A^+(y)$.

\bigskip
{\em Case 3:} If neither Case 1 nor Case 2 holds, then there are integers $n$ and $j$ such that $n(x)=n(x')=n$, and $I_n(x)=I_n(x')=I_{n,j}$. We argue that in fact, $I_{n+1}(x)=I_{n+1}(x')$. From the definition of $n(x)$ we have $s_n(x)=s_n(x')=0$. Since $x$ and $x'$ are the leftmost points of their respective equivalence classes under $\sim$, their distance from the left endpoint of $I_{n,j}$ is at most $r^{-n}$. Furthermore, since both $s_{n+1}(x)>0$ and $s_{n+1}(x')>0$, $x$ and $x'$ must both lie in the leftmost subinterval $I_{n+1,k}$ of $I_{n,j}$ on which $f_r^{n+1}$ has positive slope. Thus, $I_{n+1}(x)=I_{n+1}(x')$. 

Let $z$ be the left endpoint of $I_{n+1}(x)$, and put
\begin{equation*}
\xi:=r^n(x-z),\qquad \xi':=r^n(x'-z).
\end{equation*}
Then $\xi,\xi'\in D\cap[0,1/2)$, and by Lemma \ref{lem:self-similar-partial}, $s_k(\xi)=s_{n+k}(x)$, and $s_k(\xi')=s_{n+k}(x')$ for each $k\in\NN$. In particular, $\xi,\xi'\in D^*$, where $D^*$ was defined by \eqref{eq:D-star}. By Case 1 and Corollary \ref{cor:self-similar-limit}, there is a pair $(m,j)\in\A^+$ such that $I_{m,j}$ separates $\xi$ and $\xi'$, and
\begin{equation}
y':=r^n(y-f_r^n(z))\in f_r(I_{m,j}). 
\label{eq:y-prime}
\end{equation}
Let $I:=z+r^{-n}I_{m,j}$. Then $I=I_{n+m,k}$ for some $k$, and $I$ separates $x$ and $x'$. By Corollary \ref{cor:self-similar-limit},
\begin{equation*}
f_r(I)=f_r\left(z+r^{-n}I_{m,j}\right)=f_r^n(z)+r^{-n}f_r(I_{m,j}),
\end{equation*}
and hence, by \eqref{eq:y-prime},
\begin{equation*}
y=f_r^n(z)+r^{-n}y' \in f_r(I).
\end{equation*}
It follows that $(n+m,k)\in \A^+(y)$, and so $I=I_{n+m,k}$ has the required property.

\bigskip
In all three cases, the conclusion of the Claim follows. The proof is complete.
\end{proof}

\begin{remark}
{\rm
In view of Remark \ref{rem:only-when-odd}, Cases 1 and 3 in the above proof cannot occur when $r$ is even, so in that case the Proposition is (nearly) trivial. This illustrates once more the additional complications that arise when $r$ is odd.
}
\end{remark}

The final ingredient of the proof of Theorem \ref{thm:main}(i) is the following lemma, for which we give a probabilistic proof based on Lemma \ref{lem:CRW}. Recall that $\A^+$ is the set of pairs $(n,j)$ such that the slope of $f_r^k$ is nonnegative on $I_{n,j}$ for $k=1,\dots,n$.

\begin{lemma} \label{lem:finite-sum}
It holds that
\begin{equation*}
\sum_{(n,j)\in\A^+} \lambda(f_r(I_{n,j}))<\infty.	
\end{equation*}
\end{lemma}

\begin{proof}
Assume that on some probability space $(\Omega,\FF,\rP)$, a correlated random walk $\{S_n\}$ with steps $\{X_n\}$ and parameter $p=p_r$ is defined as in Definition \ref{def:CRW}, where $p_r$ is as in \eqref{eq:CRW-parameter}. Define a second probability measure $\rP_+$ on $(\Omega,\FF)$ by
\begin{equation*}
\rP_+(A):=2p_r\sP(A\cap\{X_1=1\})+2(1-p_r)\sP(A\cap\{X_1=-1\}), \qquad A\in\FF.
\end{equation*}
Under $\rP_+$, the process $\{S_n\}$ is a CRW with a ``flying start" in the upward direction: $\rP_+(X_1=1)=p_r$, and $\rP_+(X_n=1|X_{n-1}=1)=\sP(X_n=-1|X_{n-1}=-1)=p_r$ for $n\geq 2$.

Observe first that, for $(n,j)\in\A^+$,
\begin{equation} \label{eq:y-x-domination}
\lambda(f_r(I_{n,j}))\leq \lambda(I_{n,j}),
\end{equation}
because the graph of $f_r$ above $I_{n,j}$ consists of $r$ similar copies of the half graph $\GG_l$, and as such, its height is no greater than its width.

Define the sets
\begin{equation*}
\A_n^+:=\{j: 0\leq j<2r^{n-1},\ (n,j)\in\A^+\}, \qquad n\in\NN.
\end{equation*}
Note that for fixed $n$, the intervals $I_{n,j}, j\in\A_n^+$ are disjoint. Thus, we have
\begin{align*}
\sum_{(n,j)\in\A^+} \lambda(I_{n,j})&=\sum_{n=1}^\infty \sum_{j\in\A_n^+} \lambda(I_{n,j})
=\sum_{n=1}^\infty \lambda\left(\bigcup_{j\in\A_n^+}I_{n,j}\right)\\
&=\sum_{n=1}^\infty \lambda\{x\in[0,1): s_1(x)\geq 0, \dots, s_{n-1}(x)\geq 0, s_n(x)=0\}\\
&=\sum_{n=1}^\infty \sP(S_1\geq 0,\dots,S_{n-1}\geq 0,S_n=0),
\end{align*}
where the last equality follows by Lemma \ref{lem:CRW}. Denote the $n$th term of the last series above by $a_n$; that is,
\begin{equation*}
a_n:=\sP(S_1\geq 0,\dots,S_{n-1}\geq 0,S_n=0), \qquad n\in\NN,
\end{equation*}
and let
\begin{equation*}
b_n:=\sP(S_1>0,\dots,S_{n-1}>0,S_n=0), \qquad n\in\NN.
\end{equation*}
By a standard argument (see, for instance, \cite[p.~345]{Jain}),
\begin{align*}
b_{n+2}&=\sP(S_1=1,S_2\geq 1,\dots,S_n\geq 1,S_{n+1}=1,S_{n+2}=0)\\
&=\frac12\sP_+(S_1\geq 0,\dots,S_{n-1}\geq 0,S_n=0)\cdot p_r\\
&=p_r^2 a_n.
\end{align*}
Hence,
\begin{equation*}
\sum_{(n,j)\in\A^+} \lambda(I_{n,j})=\sum_{n=1}^\infty a_n=p_r^{-2}\sum_{n=1}^\infty b_{n+2}\leq p_r^{-2}\sP(S_n=0\ \mbox{for some $n\geq 3$})<\infty.
\end{equation*}
Together with \eqref{eq:y-x-domination}, this completes the proof.
\end{proof}

\begin{proof}[Proof of Theorem \ref{thm:main}(i)]
Since $f_r(C)$ is countable, we have, by Proposition \ref{prop:key-idea}, Lemma \ref{lem:finite-sum} and the Borel-Cantelli lemma,
\begin{align*}
\lambda\{y: |L_r(y)|=\infty\}&=\lambda\{y\in\RR\backslash f_r(C): |L_r(y)|=\infty\}\\
&\leq \lambda\{y:|\A^+(y)|=\infty\}\\
&=\lambda\{y: y\in f_r(I_{n,j})\ \mbox{for infinitely many pairs $(n,j)\in\A^+$}\}\\
&=0.
\end{align*}
Thus, $L_r(y)$ is finite for almost every $y$.
\end{proof}

\section*{Acknowledgment}

The author wishes to thank the anonymous referee for a careful reading of the paper and for suggesting several improvements to the presentation.

\end{document}